\documentclass[12pt,dvips]{amsart}
\usepackage{euler, amsfonts, amssymb, latexsym, epsfig,epic, rotating} 

\usepackage{hyperref}

\usepackage[OT2,OT1]{fontenc}
\newcommand\cyr{%
  \renewcommand\rmdefault{wncyr}%
  \renewcommand\sfdefault{wncyss}%
  \renewcommand\encodingdefault{OT2}%
  \normalfont
  \selectfont}
\DeclareTextFontCommand{\textcyr}{\cyr}

\newcommand\lie[1]{{\mathfrak #1}}

\newcommand\iso{{\,\cong\,}}
\newcommand\tensor{{\otimes}}

\newtheorem{Theorem}{Theorem} 
\newtheorem{Proposition}{Proposition} 
\newtheorem{Lemma}{Lemma}

\newtheorem{Corollary}{Corollary}
\newtheorem*{Corollary*}{Corollary}
\newtheorem*{Lemma*}{Lemma}
 
\newtheorem*{Theorem*}{Theorem}
\theoremstyle{remark}
\newtheorem{Example}{Example}
\newtheorem{Remark}{Remark}

\newcommand\reals{{\mathbb R}}
\newcommand\complexes{{\mathbb C}}
\newcommand\integers{{\mathbb Z}}

\newcommand\inv{\mathrm{inv}}
\newcommand\Inv{\mathrm{Inv}}
\newcommand\Mat{\mathrm{Mat}}
\newcommand\GL{\mathrm{GL}}
\newcommand\bkdot{\odot_0}
\newcommand\wt{\mathrm{wt}}
\newcommand\Evec{\longrightarrow}
\newcommand\NEvec{\raisebox{-1ex}{\begin{turn}{60}$\longrightarrow$\end{turn}}}
\newcommand\SEvec{\raisebox{2ex}{\begin{turn}{-60}$\longrightarrow$\end{turn}}}

\newcommand\NWvec{\raisebox{2ex}{\begin{turn}{-60}$\longleftarrow$\end{turn}}}
\newcommand\SWvec{\raisebox{-1ex}{\begin{turn}{60}$\longleftarrow$\end{turn}}}

\theoremstyle{plain}

\theoremstyle{remark}

\newtheoremstyle{TheoremNum}
{\topsep}{\topsep}              
{\itshape}                      
{}                              
{\bfseries}                     
{.}                             
{ }                             
{\thmname{#1}\thmnote{ \bfseries #3}}
\theoremstyle{TheoremNum}
\newtheorem{thmn}{Theorem}

\renewenvironment{quotation}
{\list{}{
    \setlength\itemindent{0em}%
    \setlength\leftmargin{1.5em}
    \setlength\rightmargin{1.5em}
  }%
\item[]}
{\endlist}

\newcommand\To{\longrightarrow}

\newcommand\defn[1]{{\bf #1}} 

\newcommand\junk[1]{}

\newcommand\Fl{F\ell }
\newcommand\oneornot{{1]}}
\newcommand\onecompl{{1^c}}

\newcommand\BDRY{{\tt BDRY}}
\newcommand\LRn{\BDRY(n)}
\newcommand\codim{\operatorname{codim}}

\begin{document}
\pagestyle{plain}

\title{Product and puzzle formulae \\ for $GL_n$ Belkale-Kumar coefficients}

\author{Allen Knutson}
\thanks{AK was supported by NSF grant 0303523.}
\email{allenk@math.cornell.edu}
\author{Kevin Purbhoo}
\thanks{KP was supported by an NSERC discovery grant.}
\email{kpurbhoo@math.uwaterloo.ca}
\date{\today}

\maketitle

\begin{abstract}
  The Belkale-Kumar product on $H^*(G/P)$ is a degeneration of the usual 
  cup product on the cohomology ring of a generalized flag manifold. 
  In the case $G=GL_n$, it was used by N. Ressayre to determine the 
  regular faces of the Littlewood-Richardson cone.

  We show that for $G/P$ a $(d-1)$-step flag manifold, each
  Belkale-Kumar structure constant is a product of $d\choose 2$
  Littlewood-Richardson numbers, for which there are many formulae
  available, e.g. the puzzles of [Knutson-Tao '03].  
  This refines previously known factorizations into $d-1$ factors.
  We define a new family of puzzles to assemble these to give a direct
  combinatorial formula for Belkale-Kumar structure constants.

  These ``BK-puzzles'' are related to extremal honeycombs, as in
  [Knutson-Tao-Woodward~'04]; using this
  relation we give another proof of Ressayre's result.

  Finally, we describe the regular faces of the Littlewood-Richardson
  cone on which the Littlewood-Richardson number is always $1$;
  they correspond to nonzero Belkale-Kumar coefficients on partial
  flag manifolds where every subquotient has dimension $1$ or $2$.
\end{abstract}

{\Small
  \setcounter{tocdepth}{1}
  \tableofcontents
}

\markright{\MakeUppercase{
    Product and puzzle formulae for $GL_n$ Belkale-Kumar coefficients }}

\section{Introduction, and statement of results}

Let $0 = k_0 < k_1 < k_2 < ... < k_d = n$ be a sequence of natural numbers,
and $\Fl(k_1,\ldots,k_d)$ be the space of partial flags
$\{ (0 < V_1 < \ldots < V_d = \complexes^n) : \dim V_i = k_i\}$
in $\complexes^n$.

Schubert varieties on $\Fl(k_1, \dots, k_d)$ are indexed by certain words
$\sigma = \sigma_1\dots\sigma_n$ on a totally ordered alphabet
of size $d$ (primarily, we will use $\{1,2,\dots,d\}$).  
The \defn{content} of $\sigma$ is the sequence $(n_1, n_2, \dots, n_d)$,
where $n_i$ is the number of $i$s in $\sigma$.  We associate to $\sigma$
a permutation $w_\sigma$, whose one-line notation lists the positions 
of the $1$s in order, followed by the positions of the $2$s, and so on
(e.g. $w_{12312} = 14253$;  
if $\sigma$ is the one-line notation of a permutation, i.e. $\forall i n_i=1$,
then $w_\sigma = \sigma^{-1}$).  
We say that $(p,q)$ is an \defn{inversion} 
of $\sigma$ if $p < q$, $w_\sigma(p) > w_\sigma(q)$; more specifically 
$(p,q)$ is an \defn{$ij$-inversion} if
additionally we have $\sigma_q = i > j = \sigma_p$.  Let $\inv(\sigma)$
(resp. $\inv_{ij}(\sigma)$) denote the number of inversions 
(resp. $ij$-inversions) of $\sigma$.

Given a word $\sigma$ of content $(k_1, k_2 -k_1, \dots, k_d-k_{d-1})$, and 
a complete flag $F_\bullet$, the 
\defn{Schubert variety}
$X_\sigma(F_\bullet) \subset \Fl(k_1, \dots, k_d)$ is
defined to be the closure of
\[
  \big\{ (0 < V_1 < \ldots < V_d)\, : \,
  (V_{\sigma_p} \cap F_{n-p+1}) \neq (V_{\sigma_p} \cap F_{n-p})
  \text{ for $p=1, \dots, n$}
  \big\}\,.
\]
(In many references, this is the Schubert variety associated to $w_\sigma$.)
With these conventions, the codimension of $X_\sigma(F_\bullet)$ is
$\inv(\sigma)$; hence the corresponding Schubert class,
denoted $[X_\sigma]$, lies in $H^{2\inv(\sigma)}(\Fl(k_1, \dots, k_d))$.

Let $\pi, \rho, \sigma$ be words with the above content.  The 
\defn{Schubert intersection number}
\[
  c_{\pi\rho\sigma} = 
  \int_{\Fl(k_1, \dots, k_d)} [X_\pi] [X_\rho] [X_\sigma]
\]
counts the number of points in a triple intersection
$X_\pi(F_\bullet) \cap X_\rho(G_\bullet) \cap X_\sigma(H_\bullet)$,
when this intersection is finite and transverse.
These numbers are also the structure constants of the cup product for
the cohomology ring
$H^*(\Fl(k_1,\dots, k_d)$.  Write $c_{\pi\rho}^\sigma := 
c_{\pi\rho\sigma^\vee}$, where $\sigma^\vee$ is $\sigma$ reversed.
The correspondence $[X_\sigma] \mapsto [X_{\sigma^\vee}]$ 
takes the Schubert basis to its dual under the Poincar\'e pairing, and so
\[
  [X_\pi] [X_\rho] = \sum_\sigma c_{\pi\rho}^\sigma [X_\sigma]
\]
in $H^*(\Fl(k_1, \dots, k_d))$.

\newcommand\BKc{\widetilde{c}}

We are interested in a different product structure on
$H^*(\Fl(k_1, \dots, k_d))$,
the \defn{Belkale-Kumar product} \cite{BK},
\[
  [X_\pi] \bkdot [X_\rho] 
  = \sum_\sigma \BKc_{\pi\rho}^\sigma [X_\sigma]
\]
whose structure constants can be defined as follows 
(see proposition~\ref{prop:invcondition}):
\begin{equation}
\label{eqn:bkinv}
  \BKc_{\pi\rho}^\sigma = 
  \begin{cases}
  c_{\pi\rho}^\sigma 
    &\quad\text{if $\inv_{ij}(\pi) + \inv_{ij}(\rho) = \inv_{ij}(\sigma)$
      for $1 \leq j < i \leq d$} \\
  0 & \quad\text{otherwise.}
  \end{cases}
\end{equation}
If our flag variety is a Grassmannian, this coincides with the
cup product; otherwise, it can be seen as a degenerate version.  
The Belkale-Kumar product has proven to be the more relevant product for 
describing the Littlewood-Richardson cone (recalled in \S \ref{sec:extremal}).

Our principal results are a combinatorial formula for the
Belkale-Kumar structure constants, and using this formula, a way to
factor each structure constant as a product of $d\choose 2$
Littlewood-Richardson coefficients.\footnote{%
  Since finishing this paper, we learned that N. Ressayre had been
  circulating a conjecture that some such factorization formula should exist.}
There are multiple known factorizations (such as in \cite{Richmond}) 
into $d-1$ factors, of which this provides a common refinement.

The factorization theorem is quicker to state.
For $S \subset \{1,\dots,d\}$, define the \defn{$S$-deflation} 
$D_S(\sigma)$ of $\sigma$ to be the
word on the totally ordered alphabet $S$ obtained by deleting letters
{\em not in} $\sigma$.
In particular $D_{ij}(\sigma)$ has only the letters $i$ and $j$.

\vskip.1in

\begin{thmn}[\ref{thm:fineproduct}](in \S \ref{sec:BKpuzzles})
  Let $\pi,\rho,\sigma$ be words with the same content.
  Then
  \[
  \BKc_{\pi\rho}^\sigma = 
  \prod_{i>j} c_{D_{ij} \pi, D_{ij} \rho}^{D_{ij} \sigma}\,.
  \]
\end{thmn}

The opposite extreme from Grassmannians is the case of a full flag manifold.
Then the theorem says that $[X_\pi] \bkdot [X_\rho]$ is nonzero
only if $\pi$ and $\rho$'s inversion sets are disjoint, and their
union is an inversion set of another permutation $\sigma$. 
In that case, $[X_\pi] \bkdot [X_\rho] = [X_\sigma]$,
in agreement with \cite[corollary 44]{BK} and
\cite[corollary 4]{Richmond}.

We prove this theorem by analyzing a combinatorial model for Belkale-Kumar
coefficients, which we call \emph{BK-puzzles}.\footnote{%
  In 1999, the first author privately circulated a puzzle conjecture for full
  Schubert calculus, not just the BK product, involving more puzzle
  pieces, but soon discovered a counterexample.
  The $2$-step flag manifold subcase of that conjecture seems likely
  to be true; Anders Buch has checked it up to $n=16$ (see \cite{BKT}).}
Define the two \defn{puzzle pieces} to be
\begin{enumerate}
\item A unit triangle, each edge labeled with the same letter from
  our alphabet.
\item A unit rhombus (two triangles glued together) with edges labeled
  $i,j,i,j$ where $i>j$, as in figure \ref{fig:pieces}. 
\end{enumerate}
They may be rotated in $60^\circ$ increments, but not reflected
because of the $i>j$ requirement.

\begin{figure}[ht]
  \centering
  \epsfig{file=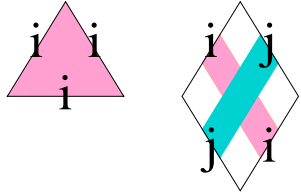,height=1in}
  \caption{The two puzzle pieces. On the rhombus, $i>j$.}
  \label{fig:pieces}
\end{figure}

A \defn{BK-puzzle} is a triangle of side-length $n$ 
filled with puzzle pieces, such that adjoining puzzle pieces have matching
edge labels. An example is in figure \ref{fig:puzex1}. We will occasionally
have need of \defn{puzzle duality}: if one reflects a BK-puzzle left-right
\emph{and} reverses the order on the labels, the result is again a BK-puzzle.

\begin{figure}[ht]
  \centering
  \epsfig{file=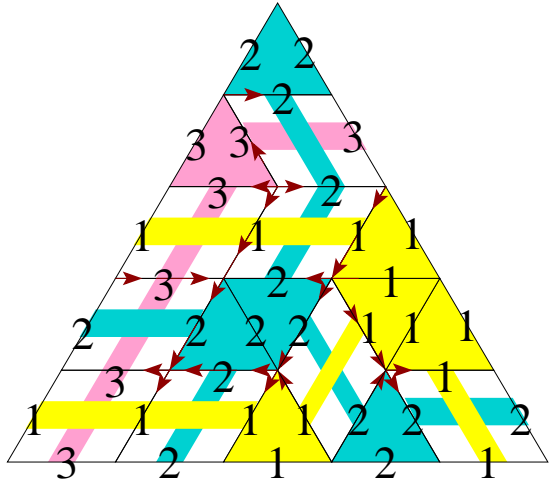,height=2.5in}
  \caption{A BK-puzzle whose existence shows
    that $\BKc_{12132,23112}^{32121}\geq 1$.
    (In fact it is $1$). 
    The edge orientations are explained in \S \ref{sec:extremal}.}
  \label{fig:puzex1}
\end{figure}

\begin{thmn}[\ref{thm:BKcountspuzzles}](in \S \ref{sec:BKpuzzles})
  The Belkale-Kumar coefficient $\BKc_{\pi\rho}^\sigma$ is
  the number of BK-puzzles with $\pi$ on the NW side, $\rho$ on the NE side,
  $\sigma$ on the S side, all read left to right.
\end{thmn}

Puzzles were introduced in \cite{KT,KTW}, where the labels were only allowed 
to be $0,1$. In this paper BK-puzzles with only two numbers will be called 
\defn{Grassmannian puzzles}. As we shall see, most of the structural
properties of Grassmannian puzzles hold for these more general BK-puzzles.
Theorem \ref{thm:assembly} corresponds a BK-puzzle to a list of
$d \choose 2$ Grassmannian puzzles, allowing us to prove theorem
\ref{thm:fineproduct} from theorem \ref{thm:BKcountspuzzles}.

Call a BK-puzzle \defn{rigid} if it is uniquely determined by its boundary,
i.e. if the corresponding structure constant is $1$.
Theorem D of \cite{RessayreGIT1}, plus the theorem above, then says
that regular faces of the Littlewood-Richardson cone (defined in
\S \ref{sec:extremal}) correspond to rigid BK-puzzles. We indicate
an independent proof of this result, and in \S \ref{sec:Fultonfaces}
determine which regular faces hold the Littlewood-Richardson coefficients
equaling $1$.

\subsection*{Acknowledgments}

We thank Shrawan Kumar for correspondence on the BK product, and
Nicolas Ressayre and Mike Roth for suggesting some references.
The honeycomb-related work was developed a number of years ago with
Terry Tao, without whom this half of the paper would have been impossible.

\section{The Belkale-Kumar product 
  on $H^*(\Fl(k_1,\ldots,k_d))$}

For the moment let $G$ be a general complex connected reductive Lie group,
and $P$ a parabolic with Levi factor $L$ and unipotent radical $N$.
Very shortly we will specialize to the $G=GL_n$ case.

\begin{Proposition}
  \label{prop:Pmovable}
  The Schubert intersection number
  $c_{\pi\rho\sigma}$ is non-zero if and only if there exist
  $a_1, a_2, a_3 \in P$
  such that
  \begin{equation}
    \label{eqn:movable}
    \lie{n} = (a_1 \lie{n}_{\pi} a_1^{-1}) \oplus (a_2 \lie{n}_{\rho} a_2^{-1})
    \oplus (a_3 \lie{n}_{\sigma} a_3^{-1})\,.
  \end{equation}
\end{Proposition}

The definition of $\lie n_\sigma$ for $G=GL_n$ will be given shortly.
Briefly, proposition~\ref{prop:Pmovable} is proven by interpreting 
$\lie{n}_\sigma$ as the conormal space at a smooth point $(V_1 < \dots < V_d)$ 
to some Schubert variety $X_\sigma(F_\bullet)$.  
The condition~\eqref{eqn:movable} measures whether it is possible to
make $(V_1 < \dots < V_d)$ a transverse point of intersection of three
such Schubert varieties.  See \cite{BK} or \cite{PurSot} for details.

Belkale and Kumar define the triple $(\pi, \rho, \sigma)$ to be
\defn{Levi-movable}
if there exist $a_1, a_2, a_3 \in L$ such that~\eqref{eqn:movable} holds.
Using this definition, they consider the numbers
\[
  \BKc_{\pi\rho\sigma} 
  := \begin{cases}
  c_{\pi\rho\sigma} &\quad\text{if $(\pi, \rho, \sigma)$ is Levi-movable} \\
  0 & \quad\text{otherwise}\,,
  \end{cases}
\]
and show that the numbers
$\BKc_{\pi\rho}^\sigma = \BKc_{\pi\rho\sigma^\vee}$
are the structure constants of a commutative, associative product
on $H^*(G/P)$. 

Our first task is to show that, in our special case $G=GL_n$, 
this definition of $\BKc_{\pi\rho}^\sigma$ is equivalent to the definition
\eqref{eqn:bkinv} given in the introduction.
In this context $P \subset \GL_n$ is the stabilizer of a coordinate flag 
$(V_1 < \dots < V_d) \in \Fl(k_1, \dots, k_d)$,
and $N \subset \GL_n$ is the unipotent Lie group with Lie algebra
\[
\lie{n} =
   \{A  \in \Mat_n : 
   A_{pq} = 0 \text{ if $p > k_{j-1}$, $q \leq k_j$  for some $j$}\} \,.
\]
We will denote the set (not the number) of all inversions of a word $\sigma$ 
(resp. $ij$-inversions) by $\Inv(\sigma)$ (resp. $\Inv_{ij}(\sigma)$).
Define
$\lie{n}_\sigma \subset \lie{n}$ to be the subspace spanned by 
$\{e_{pq} : (p,q) \in \Inv(\sigma)\}$; here
$e_{pq} \in \Mat_n$ denotes the matrix with a $1$ in row $p$, 
column $q$, and $0$s elsewhere.

\begin{Proposition}
  \label{prop:invcondition}
  The triple $(\pi, \rho, \sigma^\vee)$ is Levi-movable if and only if
  $c_{\pi\rho}^\sigma \neq 0$ and
  for all $1 \leq j < i \leq d$, we have 
  \[
  \inv_{ij}(\pi) + \inv_{ij}(\rho) = \inv_{lj}(\sigma)\,.
  \]
\end{Proposition}

\begin{proof}
  By duality (replacing $\sigma$ by $\sigma^\vee$), we may rephrase this 
  as follows. Assume 
  $c_{\pi\rho\sigma} \neq 0$.  We must show that $(\pi, \rho, \sigma)$ is
  Levi-movable iff for all $i > j$,
  \begin{equation}
    \label{eqn:symminvcondition}
    \inv_{ij}(\pi) + \inv_{ij}(\rho) + \inv_{ij}(\sigma) 
    = (k_i-k_{i-1})(k_j-k_{j-1}) \,.
  \end{equation}

The center of $L \iso \prod_{i=1}^d GL(n_i)$ is a $d$-torus, 
and acts on $\lie n$ by conjugation.
This action defines
a weight function on the standard basis for $\lie n$, which may be 
written:
$\wt(e_{pq}) = y_j - y_i$ where $k_{i-1} < q \leq k_i$
and $k_{j-1} < p \leq k_j$.
In particular, we have $\wt(e_{pq}) = y_j - y_i$ if
$(p,q)$ is an $ij$-inversion of $\pi, \rho$ or $\sigma$.  
The action of the center of $L$, and hence
the weight function, extends to the exterior algebra 
$\bigwedge^*(\lie n)$.
The weights are partially ordered: 
$\sum_{i=1}^d \alpha_i y_i$ is \defn{higher} than $\sum_{i=1}^d \beta_i y_i$
if their difference is in the cone spanned by 
$\{y_i - y_{i+1}\}_{i=1,\dots,d-1}$. 

Let $\Lambda_\pi = \bigwedge_{(p,q) \in \Inv(\pi)} e_{pq}$ and
$\Lambda_{\pi}^{ij} = \bigwedge_{(p,q) \in \Inv_{ij}(\pi)} e_{pq}$, with
$\Lambda_\rho$, $\Lambda_{\rho}^{ij}$, etc. defined analogously.
By proposition~\ref{prop:Pmovable}, there exist $a_1, a_2, a_3 \in P$ such
that
\begin{equation}
\label{eqn:Pmovablewedge}
  a_1 \Lambda_\pi a_1^{-1} \wedge 
  a_2 \Lambda_\rho a_2^{-1} \wedge 
  a_3 \Lambda_\sigma a_3^{-1}
  \neq 0
\end{equation}
Write $a_m = b_m c_m$ where $b_m \in N$, $c_m \in L$, $m = 1,2,3$.
Note that $c_m e_{pq} c_m^{-1}$ is a sum of terms of the same weight
as $e_{pq}$, and that 
\[
  a_m e_{pq} a_m^{-1} = c_m e_{pq} c_m^{-1} +
  \text{terms of higher weight.}
\]
Hence the left hand side of \eqref{eqn:Pmovablewedge} can be written as
\begin{equation}
\label{eqn:PLmovablewedge}
  c_1 \Lambda_\pi c_1^{-1} \wedge 
  c_2 \Lambda_\rho c_2^{-1} \wedge 
  c_3 \Lambda_\sigma c_3^{-1}
  + \text{terms of higher weight.}
\end{equation}
Now $\bigwedge^{\dim(\lie n)}(\lie n)$ has only one weight, which is
$\sum_{i>j} (k_i-k_{i-1})(k_j-k_{j-1}) (y_j - y_i)$.
If \eqref{eqn:symminvcondition} holds, then the first term of
\eqref{eqn:PLmovablewedge} has this weight, and the terms of higher
weight are zero; thus
\[
  c_1 \Lambda_\pi c_1^{-1} \wedge 
  c_2 \Lambda_\rho c_2^{-1} \wedge 
  c_3 \Lambda_\sigma c_3^{-1} \\
  =
  a_1 \Lambda_\pi a_1^{-1} \wedge 
  a_2 \Lambda_\rho a_2^{-1} \wedge 
  a_3 \Lambda_\sigma a_3^{-1}
  \neq 0\,,
\]
which shows that $(\pi, \rho, \sigma)$ is Levi-movable.  

Conversely,
if $(\pi, \rho, \sigma)$ is Levi-movable, then there exist 
$a_1, a_2, a_3 \in L$ such that~\eqref{eqn:Pmovablewedge} holds.
Since $\Lambda_{\pi} = \bigwedge_{i>j} \Lambda_{\pi}^{ij}$, we have
\[
  a_1 \Lambda_{\pi}^{ij} a_1^{-1} \wedge 
  a_2 \Lambda_{\rho}^{ij} a_2^{-1} \wedge 
  a_3 \Lambda_{\sigma}^{ij} a_3^{-1}
  \neq 0
\]
for all $i>j$.
Since the action of $L$ on $\lie n$ preserves the weight spaces, 
this calculation is happening inside 
$\bigwedge^*\big(\text{$y_j -y_i$ weight space of $\lie n$}\big)$.
This weight space has dimension $(k_i-k_{i-1})(k_j-k_{j-1})$, so
\[
  \inv_{ij}(\pi) + \inv_{ij}(\rho) + \inv_{ij}(\sigma) 
  \leq (k_i-k_{i-1})(k_j-k_{j-1})\,.
\]
If any of these inequalities were strict, then summing them would yield
$\inv(\pi) + \inv(\rho) + \inv(\sigma) < \dim(\Fl(k_1, \dots, k_n))$. 
But this contradicts
$c_{\pi\rho\sigma} \neq 0$, and hence we deduce~\eqref{eqn:symminvcondition}.
\end{proof}

\begin{Remark}
Belkale and Kumar also give a numerical criterion for Levi-movability
\cite[theorem 15]{BK}.  
When expressed in our notation,
their condition asserts that $(\pi, \rho, \sigma^\vee)$ is Levi-movable
iff $c_{\pi\rho}^\sigma \neq 0$ and for all $1 \leq l < d$, 
\[
  \sum_{j\leq l < i} \big(
   \inv_{ij}(\pi) + \inv_{ij}(\rho) - \inv_{ij}(\sigma)
  \big) = 0\,.
\]
It is is an interesting exercise to show combinatorially that that
this is equivalent to the condition in proposition~\ref{prop:invcondition}.
\end{Remark}

Recall from the introduction the deflation operations $D_S$ on words.
Next, consider an equivalence relation $\sim$ on $\{1, \dots, d\}$ 
such that $i\sim j, i>l>j \implies i\sim l\sim j$, and 
define $A_\sim(\sigma) := \sigma/\!\sim$ (where $A$ introduces $A$mbiguity). 

Given such an equivalence relation, 
let $S_1 < S_2 < \dots < S_{d'}$ be the (totally ordered) equivalence
classes of of $\sim$, and let $k'_i = k_{\max(S_i)}$.
There is a natural projection
\[
  \alpha_\sim : \Fl(k_1, \dots, k_d) \to \Fl(k'_1 \dots, k'_{d'})\,.
\]
whose fibres are isomorphic to products of partial flag varieties:
\[
  \Fl(k_j : j \in S_1) \times \Fl(k_j-k'_1 : j \in S_2) \times \dots \times 
  \Fl(k_j - k'_{d'-1} : j \in S_{d'})\,.
\]
The image of a Schubert variety $X_\sigma(F_\bullet)$ is the
Schubert variety $X_{A_\sim(\sigma)}(F_\bullet)$.  The fibre over
a smooth point $(V'_1 < \dots < V'_{d'})$ is a product of Schubert varieties
$X_{D_{S_1}(\sigma)} \times \dots \times X_{D_{S_{d'}}(\sigma)}$.
We will denote this fibre by $X_{D_\sim(\sigma)}(F_\bullet, V')$.

It is easy to verify that
\begin{align}
  \label{eqn:codimbase}
  \codim X_{D_\sim(\sigma)}(F_\bullet, V')
  &= \sum_{i \sim j} \inv_{ij}(\sigma) \\
  \label{eqn:codimfibre}
  \codim X_{A_\sim(\sigma)}(F_\bullet) 
  &= \sum_{i \nsim j} \inv_{ij}(\sigma) \,.
\end{align}

The main result we'll need in subsequent sections is the next lemma.
We sketch a proof here; a more detailed proof can be 
found in~\cite[Theorem 3]{Richmond}.

\begin{Lemma} 
\label{lem:coarseproduct}
Assume $(\pi, \rho, \sigma^\vee)$ is Levi-movable. Then
\[
  \BKc_{\pi\rho}^\sigma 
  = \BKc_{A_\sim(\pi)A_\sim(\rho)}^{A_\sim(\sigma)} \cdot
 \prod_{m=1}^{d'}
 \BKc_{D_{S_m}(\pi)D_{S_m}(\rho)}^{D_{S_m}(\sigma)} \,.
\]
\end{Lemma}

\begin{proof}
First observe that for generic complete flags 
$F_\bullet, G_\bullet, H_\bullet$, 
the intersections
\begin{gather}
  \label{eqn:intersecttotal}
  X_\pi(F_\bullet) \cap X_\rho(G_\bullet) \cap X_{\sigma^\vee}(H_\bullet) \\
  \label{eqn:intersectbase}
  X_{A_\sim(\pi)}(F_\bullet) \cap X_{A_\sim(\rho)}(G_\bullet) \cap 
  X_{A_\sim(\sigma^\vee)}(H_\bullet) \\
  \label{eqn:intersectfibre}
  X_{D_\sim(\pi)}(F_\bullet, V') \cap X_{D_\sim(\rho)}(G_\bullet, V') \cap 
  X_{D_\sim(\sigma^\vee)}(H_\bullet, V')
\end{gather}
are all finite and transverse.  (In~\eqref{eqn:intersectfibre}, $V'$ 
is any point of~\eqref{eqn:intersectbase}.)  
The fact that the expected dimension of each intersection is finite 
can be
seen using \eqref{eqn:codimbase}, \eqref{eqn:codimfibre} and 
proposition~\ref{prop:invcondition}.  
Transversality follows
from Kleiman's transversality theorem.  For~\eqref{eqn:intersecttotal}
and~\eqref{eqn:intersectbase}
this is a standard argument; for~\eqref{eqn:intersectfibre}, we use
the fact a Levi subgroup of the stabilizer of $V'$ acts transitively 
on the fibre $\alpha_\sim^{-1}(V')$.
%
%
%

This shows that the number of points in 
\eqref{eqn:intersecttotal} is the product of the numbers of points
in \eqref{eqn:intersectbase} and \eqref{eqn:intersectfibre}, i.e.
\[
  c_{\pi\rho}^{\sigma} 
  = c_{A_\sim(\pi)\,A_\sim(\rho)}^{A_\sim(\sigma)} \cdot
 \prod_{m=1}^{d'}
 c_{D_{S_m}(\pi)\,D_{S_m}(\rho)}^{D_{S_m}(\sigma)}\,.
\]
Again, using proposition~\ref{prop:invcondition}, we find that 
$(A_\sim(\pi), A_\sim(\rho), A_\sim(\sigma^\vee))$ and
$(D_{S_m}(\pi),D_{S_m}(\rho),D_{S_m}(\sigma^\vee))$, $m=1,\dots, d'$,
are Levi-movable;
hence we may add tildes everywhere.
\end{proof}

\section{BK-puzzles and their disassembly}\label{sec:BKpuzzles}

Say that two puzzle pieces in a BK-puzzle $P$ of exactly the same
type, and sharing an edge, are \defn{in the same region}, and let the
decomposition into \defn{regions} be the transitive closure thereof.
Each region is either made of $(i,i,i)$-triangles, and called an
\defn{$i$-region}, or $(i,j,i,j)$-rhombi, and called an
\defn{$(i,j)$-region}.

The basic operation we will need on BK-puzzles is ``deflation'' 
\cite[\S 5]{KTW}, extending the operation $D_S$ defined in the introduction on
words.

\begin{Proposition}\label{prop:deflation}
  Let $P$ be a BK-puzzle, and $S$ a set of edge labels.
  Then one can shrink all of $P$'s edges with labels {\em not in} $S$ to points,
  and obtain a new BK-puzzle $D_S P$ whose sides have been $S$-deflated.
\end{Proposition}

\begin{proof}
  It is slightly easier to discuss the case $S^c = \{s\}$, and obtain the
  general case by shrinking one number $s$ at a time.

  Let $t\in [0,1]$, and change the puzzle regions as follows: keep the
  angles the same, but shrink any edge with label $s$ to have length $t$.
  (This wouldn't be possible if e.g. we had triangles with labels
  $s,s,j\neq s$, but we don't.) For $t=1$ this is the original BK-puzzle $P$, 
  and for all $t$ the resulting total shape is a triangle.  Consider now the
  BK-puzzle at $t=0$: all the $s$-edges have collapsed, and each $(i,s)$-
  or $(s,i)$-region has shrunk to an interval, joining two $i$-regions
  together.
\end{proof}

Call this operation the \defn{$S$-deflation} $D_S P$ of the BK-puzzle $P$.

\begin{Proposition}\label{prop:green}
  Let $P$ be a BK-puzzle. Then the content $(n_1,n_2,\ldots,n_d)$ on
  each of the three sides is the same.
  There are $n_i + 1\choose 2$ right-side-up $i$-triangles and
  $n_i\choose 2$ upside-down $i$-triangles, and $n_i n_j$ $(i,j)$-rhombi,
  for all $i$ and $j$.

  More specifically, the number of $(i,j)$-rhombi (for $i>j$) with a
  corner pointing South equals the number of $ij$-inversions on the
  South side. (Similarly for NW or NE.)
\end{Proposition}

\begin{proof}
  Deflate all numbers except for $i$, resulting in a triangle of size $n_i$,
  or all numbers except for $i$ and $j$, 
  resulting in a Grassmannian puzzle. 
  Then invoke \cite[proposition 4]{KTW} and \cite[corollary 2]{KT}.
\end{proof}

Now fix $\pi,\rho,\sigma$ of the same content,
and let $\Delta_{\pi\rho}^\sigma$ denote the set of BK-puzzles with
$\pi,\rho,\sigma$ on the NW, NE, and S sides respectively, 
all read left to right. Then $D_S$ on BK-puzzles is a map
$$ D_S : \Delta_{\pi\rho}^\sigma \to \Delta_{D_S \pi, D_S \rho}^{D_S \sigma}. $$

\begin{Corollary}\label{cor:noBKs}
  Let $\pi,\rho,\sigma$ be three words. If they do not have the
  same content, then $\Delta_{\pi\rho}^\sigma = \emptyset$.
  If they have the same content, but for some $i>j$ we have
  $inv_{ij}(\pi) + inv_{ij}(\rho) \neq inv_{ij}(\sigma)$,
  then $\Delta_{\pi\rho}^\sigma = \emptyset$.
\end{Corollary}

It is easy to see that any ambiguator $A_\sim$ extends to a map
$$ A_\sim : \Delta_{\pi\rho}^\sigma 
        \to \Delta_{A_\sim \pi, A_\sim \rho}^{A_\sim \pi} $$
which one does not expect to be $1:1$ or onto in general.
The only sort we will use is ``$A_{i]}$'', which
amalgamates all numbers $\leq i$, and all numbers $> i$. 
In particular, each $A_{i]} P$ is a Grassmannian puzzle.

We will need to study a deflation (of the single label $1$)
and an ambiguation together:
$$ A_\oneornot \times D_\onecompl : \Delta_{\pi\rho}^\sigma \to 
\Delta_{A_\oneornot\pi, A_\oneornot\rho}^{A_\oneornot\sigma} \times
\Delta_{D_\onecompl \pi, D_\onecompl \rho}^{D_\onecompl \sigma}. $$
Our key lemma (lemma \ref{lem:AD}) will be that either this map is an
isomorphism or the source is empty. That suggests that we try to define
an inverse map, but to a larger set.

Define the set $(\Delta^1)_{\pi\rho}^\sigma$ of \defn{BK$^1$-puzzles}
to be those made of the following labeled pieces, plus the stipulation that 
only single numbers (not multinumbers like $(53)$) may appear on the boundary
of the puzzle triangle:

\centerline{\epsfig{file=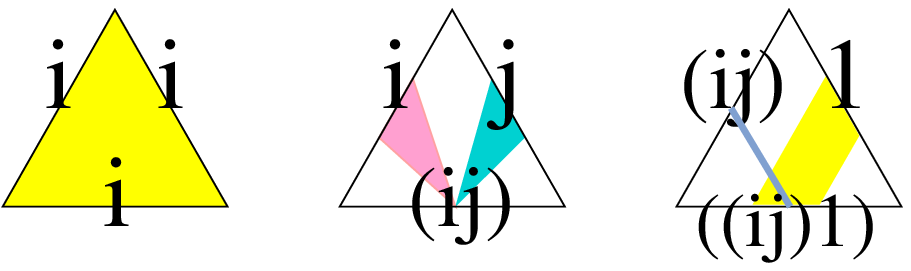,height=0.8in}}

\noindent 
Again $i>j$, and on the third pieces $i>j>1$.
If we disallow the $((ij)1)$ labels (and with them, the third type of piece),
then any triangle of the second type must be matched to another such, 
and we recover an equivalent formulation of $\Delta_{\pi\rho}^\sigma$.
In this way there is a natural inclusion 
$\Delta_{\pi\rho}^\sigma \to (\Delta^1)_{\pi\rho}^\sigma$,
cutting each $(i,j)$-rhombus into two triangles of the second type.
In figure \ref{fig:BK1ex} we give an example of a BK$^1$-puzzle that
actually uses the $((32)1)$ label.

\begin{figure}[ht]
  \centering
  \epsfig{file=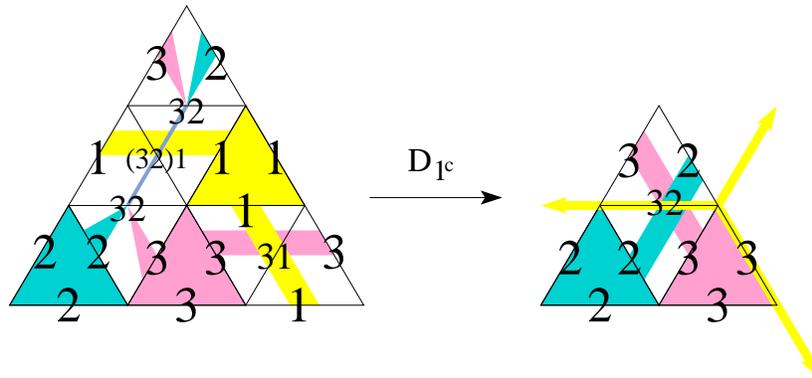,height=2in}
  \caption{The BK$^1$-puzzle on the left deflates to the Grassmannian
    puzzle on the right, which naturally carries a honeycomb
    remembering where the $1$-edges were, 
    as in the proof of lemma \ref{lem:AD}.}
  \label{fig:BK1ex}
\end{figure}

\begin{Lemma}\label{lem:invcount}
  For a word $\tau$, let $Y(\tau) := \sum_i inv_{i1}(\tau) y_i$.
  Then for each puzzle $P \in (\Delta^1)_{\pi\rho}^\sigma$,
  $$ \sum_{e\ \text{ labeled }((ij)1)} y_j - y_i 
  \quad =\quad  Y(\pi) + Y(\rho) - Y(\sigma). $$
\end{Lemma}

\begin{proof}
  Consider the vector space $\reals^2 \tensor \reals^d$, where $\reals^2$
  is the plane in which our puzzles are drawn, and $\reals^d$ has
  basis $\{x_1, \dots, x_d\}$.
  Assign to each directed edge $e$ of $P$ a vector
  $v_e \in \reals^2 \tensor \reals^d$, as follows:
  \begin{align*}
    e = \overset{i}{\longrightarrow}
     \qquad&\implies\qquad v_e = \Evec \tensor x_i \\
    e = \overset{(ij)}{\Evec}
     \qquad&\implies\qquad v_e = (\NEvec \tensor x_i) + (\SEvec \tensor x_j) \\
    e = \overset{((ij)1)}{\Evec} 
     \qquad&\implies\qquad v_e = (\NWvec \tensor x_i) + (\Evec \tensor x_j)
                 + (\SWvec \tensor x_1)
  \end{align*}
  If $e$ points in another direction, $v_e$ is rotated accordingly
  ($e \mapsto v_e$ is rotation-equivariant).
  This assignment has the property that if the edges $e,f,g$ of
  a puzzle piece are directed to to form a cycle, then $v_e + v_f + v_g = 0$.
  Consider the bilinear form $\boxdot$ on 
  $\reals^2 \tensor \reals^d$ satisfying:
  \begin{itemize}
     \item $\boxdot$ is rotationally invariant;
     \item $(e \tensor x_i) \boxdot (f \tensor x_j) = 0$, 
            if $i=1$ or $j \neq 1$;
     \item $(\Evec \tensor x_i) \boxdot (\Evec \tensor x_1) = 
           (\Evec \tensor x_i) \boxdot (\SEvec \tensor x_1) = y_i$, 
           if $i \neq 1$.
  \end{itemize}
  These conditions completely determine $\boxdot$.  
  For example, bilinearity and rotational invariance give
  \begin{align*}
  (\NEvec \tensor x_i) \boxdot (\SEvec \tensor x_1) 
   &= 
  (\Evec \tensor x_i) \boxdot (\SWvec \tensor x_1)  
   =
  (\Evec \tensor x_i) \boxdot ((\SEvec - \Evec) \tensor x_1)  \\
   &= 
  (\Evec \tensor x_i) \boxdot (\SEvec \tensor x_1)  -
  (\Evec \tensor x_i) \boxdot (\Evec \tensor x_1)  
   = y_i - y_i = 0 \,.
  \end{align*}

  Let $\Omega = (e_1, \dots, e_m)$ be a path from the southwest corner 
  of $P$ to the southeast corner.
  Let $Y(\Omega) := \sum_{r<s} v_{e_r} \boxdot v_{e_s}$.  We claim the 
  following:
  \begin{enumerate}
    \item If $\Omega$ is the path along the south side, then 
          $Y(\Omega) = Y(\sigma)$.
    \item If $\Omega$ is the path that goes up the northwest side and down the
          northeast side, then $Y(\Omega) = Y(\pi) + Y(\rho)$.
    \item If $\Omega$ is any path, then 
          $Y(\Omega) = Y(\sigma) + \sum_e y_j - y_i$
          where the sum is taken over edges $e$ labeled $((ij)1)$, lying 
          strictly below $\Omega$.
  \end{enumerate}
  The first two assertions are easily checked (the second uses the 
  calculation in the example above).  For the third, we proceed by 
  induction on the number of puzzle pieces below $\Omega$.  
  We show that if we alter the path so as to
  add a single puzzle piece, $Y(\Omega)$ doesn't change, except 
  when new the piece 
  is attached to an edge of $\Omega$ labeled $((ij)1)$, in which case 
  it changes 
  by $y_j - y_i$.  To see this, note that when a piece is added, the 
  sequence $(v_{e_1}, \dots, v_{e_m})$ changes in a very simple way:
  either two consecutive vectors $v_{e_r}$ and $v_{e_{r+1}}$ are replaced 
  by their sum, or the
  reverse---a vector $v_{e_r}$ in the sequence is replaced by two 
  consecutive vectors $v_{e_r'}, v_{e_r''}$ with sum $v_{e_r}$.  When the 
  first happens, $Y(\Omega)$ changes by $- v_{e_r} \boxdot v_{e_{r+1}}$; in
  the second case, by $ v_{e_r'} \boxdot v_{e_r''}$.  It is now a
  simple matter to check that this value is $0$ or $y_j - y_i$ as
  indicated.

  The lemma follows from assertions 1 and 3 in the claim.
\end{proof}

\begin{Lemma}\label{lem:AD}
  Fix $\pi,\rho,\sigma$, and let $c = Y(\pi) + Y(\rho) - Y(\sigma)$ be
  the statistic from lemma \ref{lem:invcount}.

  \begin{enumerate}
  \item 
    The map $A_\oneornot \times D_\onecompl$ defined above factors as
    the inclusion $\Delta_{\pi\rho}^\sigma \to (\Delta^1)_{\pi\rho}^\sigma$
    followed by a bijection $(\Delta^1)_{\pi\rho}^\sigma
    \ \widetilde{\longrightarrow}\
    \Delta_{A_\oneornot\pi, A_\oneornot\rho}^{A_\oneornot\sigma} \times
    \Delta_{D_\onecompl \pi, D_\onecompl \rho}^{D_\onecompl \sigma}$.
  \item If $c=0$, then the inclusion
    $\Delta_{\pi\rho}^\sigma \to (\Delta^1)_{\pi\rho}^\sigma$
    is an isomorphism.
  \item If $c\neq 0$, then $\Delta_{\pi\rho}^\sigma$ is empty.
    If $c$ is not in the cone spanned by $\{y_i - y_{i+1}\}_{i=1,\ldots,d-1}$, 
    then $(\Delta^1)_{\pi\rho}^\sigma$ and at least one of
    $\Delta_{A_\oneornot\pi, A_\oneornot\rho}^{A_\oneornot\sigma}$,
    $\Delta_{D_\onecompl \pi, D_\onecompl \rho}^{D_\onecompl \sigma}$ is empty.
  \end{enumerate}
\end{Lemma}

To define the reverse map will require the concept of ``honeycomb''
from \cite{Hon1}.

Consider a ``multiplicity'' function from the $3{n+1 \choose 2}$ edges of
a triangle of size $n$ to the naturals, and define the ``tension'' on a
vertex to be the vector sum of its incident edges (thought of as
outward unit vectors), weighted by these multiplicities. 
Define a \defn{bounded honeycomb} as a multiplicity function such that 
\begin{itemize}
\item the tension of internal vertices is zero,
\item the tension of vertices on the NW edge (except the North corner)
  is horizontal, and
\item the $120^\circ,240^\circ$ rotated statements hold for the 
  remainder of the boundary.
\end{itemize}
One can add two such multiplicity functions, giving an additive
structure on the set of bounded honeycombs of size $n$, called \defn{overlay} 
and denoted $\oplus$ (as it is related to the direct sum operation on
Hermitian matrices \cite{Hon1}).
The \defn{dimension} of a bounded honeycomb is the sum of the multiplicities
on the horizontal edges meeting the NW edge, and is additive under $\oplus$.
A bounded honeycomb has \defn{only simple degeneracies} if all multiplicities
are $1$, and there are no vertices of degree $>4$. If it also has no vertices of
degree $4$, it is \defn{generic}.

If one relieves the tension on the boundary by attaching rays to infinity,
one obtains (a $30^\circ$ rotation of) a \defn{honeycomb} as defined in
\cite{Hon1}. (In this way one can see that the dimension of a bounded honeycomb
is invariant under rotation.)
Bounded honeycombs already arose in \cite[\S 5, theorem 1]{KTW}.

\begin{proof}[Proof of lemma \ref{lem:AD}]
  (1) It is easy to extend $A_\sim$ to a map 
    $ (\Delta^1)_{\pi\rho}^\sigma \to
    \Delta_{A_\oneornot\pi, A_\oneornot\rho}^{A_\oneornot\sigma}$,
    by \\
    \centerline{\epsfig{file=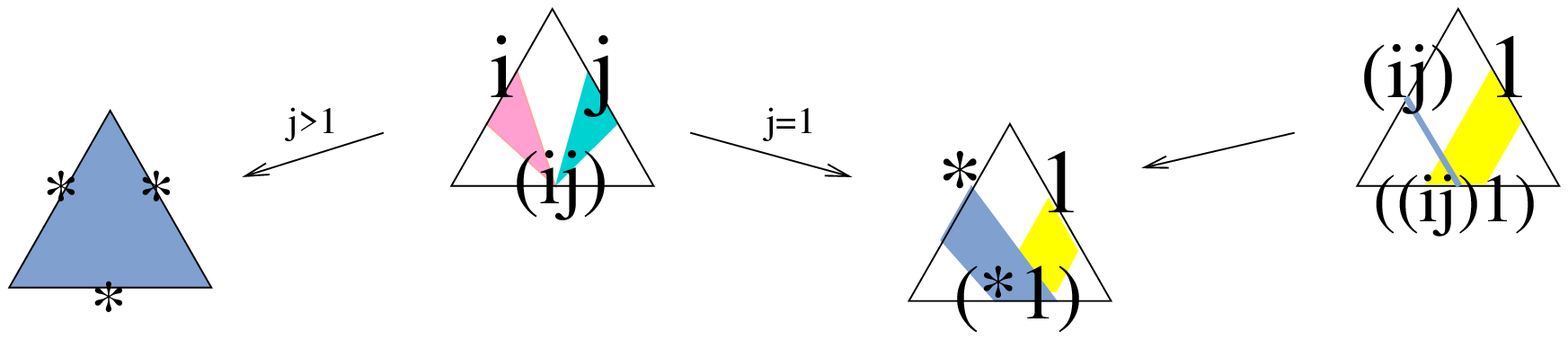,height=1.2in}}
    where $*$ represents the equivalence class ``numbers larger than $1$''.

    To extend $D_\onecompl$ to a map 
    $ (\Delta^1)_{\pi\rho}^\sigma \to \Delta_{D_\onecompl \pi, D_\onecompl \rho}^{D_\onecompl \sigma}$,
    first erase any puzzle edge that has an $(i1)$ or $((ij)1)$ on it, which
    results in a decomposition into triangles and rhombi. 
    Then as before, shrink the $1$-edges.

    Given a pair $(G,P) \in
    \Delta_{A_\oneornot\pi, A_\oneornot\rho}^{A_\oneornot\sigma} \times
    \Delta_{D_\onecompl \pi, D_\onecompl \rho}^{D_\onecompl \sigma}$, we know 
    $D_\onecompl G$ and $P$ are the same size, and $D_\onecompl G$ is a trivial 
    puzzle in the sense that all edges are labeled $*$. But this triangle
    has more structure, if one remembers where the deflated $(*,1)$-rhombi
    are in it: it has a bounded honeycomb of dimension equal to the number
    of $1$s in each of $\pi,\rho,\sigma$. 
    (This is the content of \cite[\S 5, theorem 1]{KTW}.)
    An example is in figure \ref{fig:BK1ex}.

    Take this bounded honeycomb on $D_\onecompl G$ and overlay it on $P$.
    Then ``reinflate'' $P$ to a BK$^1$-puzzle $Q$ with $D_\onecompl Q = P$,
    where each honeycomb edge of multiplicity $m$ and $P$-label $I$
    (which may be $i$ or $(ij)$) is inflated to $m$ $(I,1)$ rhombi
    with label $(I1)$ across their waists. This is the inverse map.

    (2) If $c=0$, by lemma \ref{lem:invcount} 
    each $P \in (\Delta^1)_{\pi\rho}^\sigma$ has no $((ij)1)$ labels,
    and hence is in the image of the inclusion
    $\Delta_{\pi\rho}^\sigma \to (\Delta^1)_{\pi\rho}^\sigma$.

    (3) If $c\neq 0$, by lemma \ref{lem:invcount} 
    every $P \in (\Delta^1)_{\pi\rho}^\sigma$ has some $((ij)1)$ labels,
    and hence is not in the image of the inclusion
    $\Delta_{\pi\rho}^\sigma \to (\Delta^1)_{\pi\rho}^\sigma$.
    So $\Delta_{\pi\rho}^\sigma = \emptyset$.

    By lemma \ref{lem:invcount}, if $P$ is in $(\Delta^1)_{\pi\rho}^\sigma$,
    then $c$ is in the positive span of $\{y_i - y_{i+1}\}_{i=1,\ldots,n-1}$.
    Contrapositively, if $c$ is not in this span, there can be no such $P$.
    Since the empty set $(\Delta^1)_{\pi\rho}^\sigma$ is isomorphic to
    $\Delta_{A_\oneornot \pi, A_\oneornot\rho}^{A_\oneornot\sigma} \times
    \Delta_{D_\onecompl \pi, D_\onecompl \rho}^{D_\onecompl \sigma}$, 
    one of these factors must be empty.
\end{proof}

\begin{Theorem}
  \label{thm:BKcountspuzzles}
  Suppose that $\pi, \rho, \sigma$ have the same content.
  Then the BK coefficient $\BKc_{\pi\rho}^\sigma$ 
  equals the cardinality $|\Delta_{\pi\rho}^\sigma|$.
\end{Theorem}

\begin{proof}
If $\inv_{ij}(\pi) + \inv_{ij}(\rho) \neq \inv_{ij}(\sigma)$ for some
$i > j$, then $\BKc_{\pi\rho}^\sigma = 0 = |\Delta_{\pi\rho}^\sigma|$,
by proposition~\ref{prop:invcondition} and corollary~\ref{cor:noBKs}.

If $\inv_{ij}(\pi) + \inv_{ij}(\rho) = \inv_{ij}(\sigma)$ for all
 $i > j$, the statistic from lemma~\ref{lem:invcount} is zero,
and we proceed by induction on $d$.  If $d=2$,
this is the Grassmannian case, where the result is known, so assume
$d \geq 3$. From lemma~\ref{lem:AD}, we have
\[
   \Delta_{\pi\rho}^\sigma
    \ \widetilde{\longrightarrow}\
    \Delta_{A_\oneornot\pi, A_\oneornot\rho}^{A_\oneornot\sigma} \times
    \Delta_{D_\onecompl \pi, D_\onecompl \rho}^{D_\onecompl \sigma}\,.
\]
By lemma~\ref{lem:coarseproduct} (applied to the equivalence relation
``\oneornot'') we have
\[
   \BKc_{\pi\rho}^\sigma
    =
    \BKc_{A_\oneornot\pi, A_\oneornot\rho}^{A_\oneornot\sigma} \cdot
    \BKc_{D_\onecompl \pi, D_\onecompl \rho}^{D_\onecompl \sigma}\,.
\]
Inducting on $d$, we have
$\BKc_{A_\oneornot\pi, A_\oneornot\rho}^{A_\oneornot\sigma}
= |\Delta_{A_\oneornot\pi, A_\oneornot\rho}^{A_\oneornot\sigma}|$, 
$\BKc_{D_\onecompl \pi, D_\onecompl \rho}^{D_\onecompl \sigma} =
|\Delta_{D_\onecompl \pi, D_\onecompl \rho}^{D_\onecompl \sigma}|$
and so 
$\BKc_{\pi\rho}^\sigma = 
   |\Delta_{\pi\rho}^\sigma|$
as required.
\end{proof}  

\begin{Theorem}\label{thm:assembly}
  Let $\pi,\rho,\sigma$ be words with the same content, and
  for all $i>j$, $inv_{ij}(\pi) + inv_{ij}(\rho) = inv_{ij}(\sigma)$.
  Amalgamate the maps $D_{ij} : \Delta_{\pi\rho}^\sigma \to
  \Delta_{D_{ij} \pi, D_{ij} \rho}^{D_{ij} \sigma} $ into a single map
  $$ (D_{ij}) : \Delta_{\pi\rho}^\sigma 
  \To \prod_{i>j} \Delta_{D_{ij} \pi, D_{ij} \rho}^{D_{ij} \sigma}. $$
  Then this map is a bijection.
\end{Theorem}

An example of the image is in figure \ref{fig:deflations}.

\begin{figure}
  \centering
  \epsfig{file=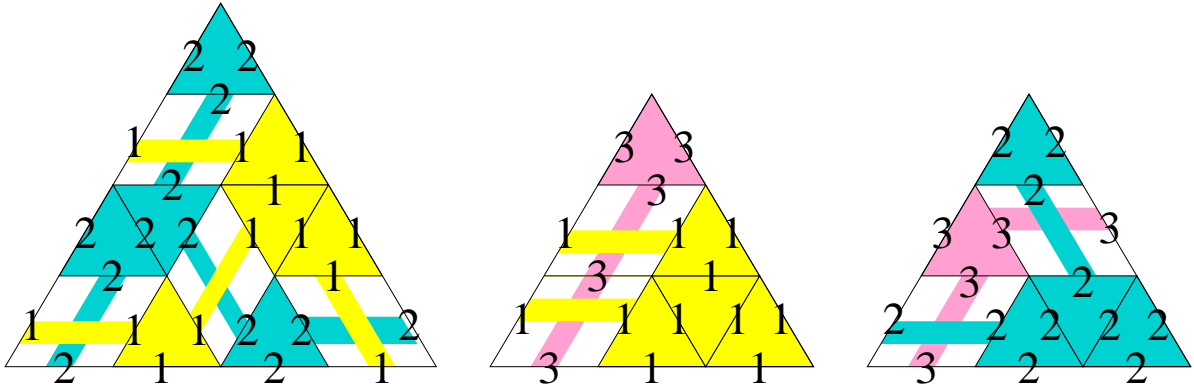,height=2in}
  \caption{The $3\choose 2$ deflations of the BK-puzzle from 
    figure \ref{fig:puzex1}, from which it may be reassembled 
    as in theorem \ref{thm:assembly}.}
  \label{fig:deflations}
\end{figure}

\begin{proof}
  If the number of labels is $1$ or $2$, the statement is trivial.
  So assume it is at least $3$.

  We need to define the reverse map, 
  associating a puzzle $Q\in \Delta_{\pi\rho}^\sigma$
  to a tuple 
  $(G_{ij} \in \Delta_{D_{ij} \pi, D_{ij} \rho}^{D_{ij} \sigma})_{i>j}$. 
  By induction on the number of labels,
  we may assume that there exists a puzzle $P$ (with no $1$s on its boundary)
  mapping to the tuple $(G_{ij})_{i>j>1}$. 
  Then the desired $Q$ must have $D_\onecompl Q = P$.

  Next we define a Grassmannian puzzle $G$ from the tuple $(G_{i1})_{i>1}$. 
  Each $D_1 G_{i1}$ is a triangle of the same size, all edges labeled $1$,
  but bearing a bounded honeycomb $h_i$ as in the proof of lemma \ref{lem:AD}.
  Then $\oplus_{i>1} h_i$ is a bounded honeycomb in this same triangle.
  Again as in the proof of lemma \ref{lem:AD}, inflate $\oplus_{i>1} h_i$ 
  to produce a puzzle $G$ with two labels $* > 1$. With this we similarly
  constrain $Q$, by $A_\oneornot Q = G$.

  Now use lemma \ref{lem:AD} parts (1) and (2) to construct $Q$
  from the pair $(P,G)$. 
\end{proof}

\begin{Theorem}\label{thm:fineproduct}
For all $\pi, \rho, \sigma$ as in theorem~\ref{thm:BKcountspuzzles},
\[
  \BKc_{\pi\rho}^\sigma = 
  \prod_{i>j} c_{D_{ij} \pi, D_{ij} \rho}^{D_{ij} \sigma}\,.
\]
\end{Theorem}

\begin{proof} This follows immediately from 
theorems~\ref{thm:BKcountspuzzles} and~\ref{thm:assembly}.
\end{proof}

\begin{Remark}
It has been observed several times now (e.g. \cite{DW,KTT})
that when a Horn inequality
is satisfied with equality, a Littlewood-Richardson number factors.  
This fact can be seen from theorem~\ref{thm:fineproduct} as follows.

Let $\pi, \rho, \sigma \in \{1,2,3\}^*$ be words of length $n$. 
Let $\overline{\pi} = A_{2]}\pi$
and $\pi' = D_{12}(\pi)$, with similar notation for $\rho$, $\sigma$,
and assume that $c_{\pi'\rho'}^{\sigma'} \neq 0$.

Under these conditions, $\inv_{ij}(\pi) + \inv_{ij}(\rho) =
\inv_{ij}(\sigma)$ for all $i > j$ is equivalent to asserting that the
Horn inequality for $(\overline\pi, \overline\rho, \overline\sigma)$
associated to $(\pi',\rho',\sigma')$ holds with equality 
(see~\cite[\S 4]{PurSot}).  With this hypothesis, the Littlewood-Richardson
number $c_{\overline\pi \overline\rho}^{\overline\sigma}$ factors as
\[
 c_{\overline\pi \overline\rho}^{\overline\sigma}
   = c_{D_{23}\pi,D_{23}\rho}^{D_{23}\sigma} \cdot
    c_{D_{13}\pi,D_{13}\rho}^{D_{13}\sigma} \cdot\,.
\]
This can be seen by comparing the result of 
lemma~\ref{lem:coarseproduct}
\[
  \BKc_{\pi\rho}^\sigma 
  = c_{\overline\pi \overline\rho}^{\overline\sigma}
  \cdot c_{\pi' \rho'}^{\sigma'} 
\]
and theorem \ref{thm:fineproduct}
\[
  \BKc_{\pi\rho}^\sigma 
  = c_{D_{23}\pi, D_{23}\rho}^{D_{23}\sigma} \cdot
    c_{D_{13}\pi,D_{13}\rho}^{D_{13}\sigma} \cdot
    c_{\pi' \rho'}^{\sigma'}\,.
\]
\end{Remark}

\section{Relation to extremal honeycombs}\label{sec:extremal}

Let $\reals^n_+$ denote the cone of weakly decreasing $n$-tuples.
The \defn{Littlewood-Richardson cone} $\LRn \subset (\reals^n_+)^3$ is a 
rational polyhedral cone whose elements $(\lambda,\mu,\nu)$ can be
characterized in several ways including the following \cite{Fulton,KTW}:
\begin{itemize}
\item There exists a triple $(H_\lambda,H_\mu,H_\nu)$ of Hermitian matrices
  of size $n$, adding to zero, whose spectra are $(\lambda,\mu,\nu)$.
\item There exists a honeycomb $h$ of size $n$ whose boundary edges $\partial h$
  have constant coordinates given by $(\lambda,\mu,\nu)$.
  (We will not use this characterization, and refer the interested reader
  to \cite{Hon1} for definitions.)
\item (If $(\lambda,\mu,\nu)$ are integral, hence may be thought of
  as dominant weights for $GL_n(\complexes)$.) 
  There is a $GL_n(\complexes)$-invariant vector in the tensor product
  $V_\lambda \tensor V_\mu \tensor V_\nu$ of irreducible representations
  with those high weights.
\item For each puzzle $P$ of size $n$, with boundary labels $0,1$, the 
  inequality $NW(P) \cdot \lambda + NE(P) \cdot \mu + S(P) \cdot \nu \leq 0$
  holds, where $NW(P),NE(P),S(P)$ are the vectors of $0$s and $1$s around the
  puzzle \emph{all read clockwise} and $\cdot$ is the dot product.
\end{itemize}
The fourth says that each inequality defining $\LRn$ (other than the
\defn{chamber inequalities} that say $\lambda,\mu,\nu$ are decreasing)
can be ``blamed'' on a Grassmannian puzzle.  In \cite{KTW} it was shown
that the puzzles that occur this way are exactly the \defn{rigid}
ones, meaning that they are uniquely determined by their boundaries.
(The others define valid, but redundant, inequalities.)

In this section we extend this last connection to one between all 
\defn{regular} faces of $\LRn$ (meaning, not lying on chamber walls)
and rigid BK-puzzles. Then the connection between BK-puzzles and the
BK product gives a new proof of \cite[Theorem D]{RessayreGIT1},
corresponding the regular faces to BK coefficients equaling $1$.

This section closely follows \cite[\S 3 and \S 4]{KTW}, and we will
only point out where the proofs there need other than trivial modification.
In any case, much of it can be avoided by {\em invoking} 
\cite[Theorem D]{RessayreGIT1}, rather than reproving it combinatorially.

\begin{Lemma}\label{lem:lem3}  (Extension of \cite[lemma 3]{KTW}.)
  Let $b$ be a generic point on a regular face of $\LRn$ of codimension 
  $d-1$. Then there exists a honeycomb $h$ with $\partial h = b$ such that
  \begin{itemize}
  \item $h = h_1 \oplus h_2 \oplus \cdots \oplus h_d$
    is an overlay of generic honeycombs,
  \item $h$ has simple degeneracies, and
  \item $h_i$ intersects $h_j$ transversely for each $j<i$, 
    and at each point $p$ where $h_i$ crosses $h_j$, 
    $h_j$ \defn{turns clockwise to} $h_i$, meaning that a path 
    going from $h_j$ to $h_i$ through $p$ could turn right $60^\circ$,
    not left $60^\circ$ (where continuing straight is turning $0^\circ$).
  \end{itemize}
\end{Lemma}

If $h = h_1 \oplus h_2 \oplus \cdots \oplus h_d$ satisfies this
third condition, call it a \defn{clockwise overlay}.

\begin{proof}
  The proof of \cite[lemma 3]{KTW} goes by showing that a honeycomb of size $m$
  with simple degeneracies that is \emph{not} a clockwise overlay has
  a $(3m-1)$-dimensional space of perturbations. So if $h$ is written
  as a clockwise overlay of fewer than $d-1$ honeycombs (e.g. as itself), 
  one of them must itself be a clockwise overlay.
\end{proof}

This already implies the interesting fact that while $\LRn$ has faces
of all dimensions $2,\ldots,3n-1$, its regular faces are of dimension
at least $2n$. 

\begin{Lemma}\label{lem:lem4} (Extension of \cite[theorem 2 and lemma 4]{KTW}.)
  Let $h = h_1 \oplus h_2 \oplus \cdots \oplus h_d$ be a clockwise overlay.
  Then there is a codimension $d-1$ regular face $F$ of $\LRn$ containing
  $\partial h$. It is the intersection of $d-1$ regular facets.

  Moreover, one can construct from $h$ a BK-puzzle $P$ with $d$ labels,
  and for each label $i<d$ one can construct a Grassmannian puzzle
  $A_{i]} P$, such that $F$ is the intersection
  of the facets corresponding to $(A_{i]} P)$.
\end{Lemma}

The construction of $P$ from $h$ is straightforward:
each $3$-valent vertex in $h_i$ is replaced with an $i$-triangle, and 
each $4$-valent vertex is replaced either with two $i$-triangles (if the vertex
lies only on $h_i$) or an $(i,j)$-rhombus (if the vertex is where
$h_i$ and $h_j$ cross). The clockwise condition causes the puzzle rhombi to
have the required $i>j$ condition. An example is in figure \ref{fig:overlayex}.

\begin{figure}[ht]
  \centering
  \epsfig{file=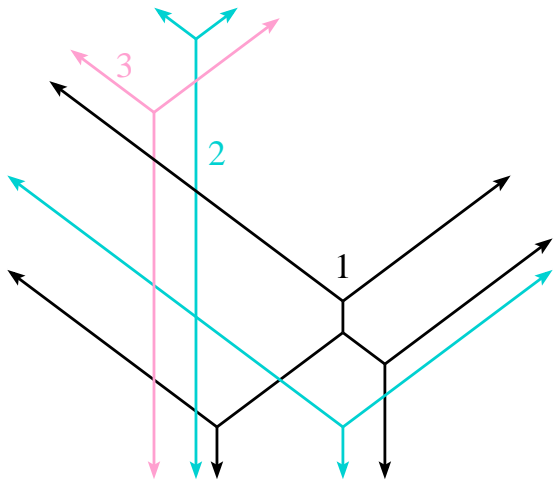,height=2.5in}
  \caption{A clockwise overlay $h_1\oplus h_2\oplus h_3$, whose associated
    BK-puzzle is the one in figure \ref{fig:puzex1}.
    The Grassmannian puzzles associated to 
    $h_1\oplus h_2, h_1\oplus h_3, h_2\oplus h_3$ are in
    figure \ref{fig:deflations}.}
  \label{fig:overlayex}
\end{figure}

\begin{Remark}
Borrowing terminology from the study of toric varieties, call a face $F$
of a polytope \defn{rationally smooth} if $F$ lies on only $\codim F$ many 
facets. Then lemma \ref{lem:lem4} implies the curious fact that $\LRn$
is rationally smooth on all its regular faces. (Note that this makes it easy to
describe the partial order on faces, as is done in 
\cite[Theorem D]{RessayreGIT2}.) This does not follow from
its interpretation as a moment polytope (see e.g. \cite[appendix]{KTW});
while the moment polytope of a full flag manifold has all rationally
smooth faces, the moment polytope of $Gr_2(\complexes^4)$ is an octahedron.
\end{Remark}

Combining lemmas \ref{lem:lem3} and \ref{lem:lem4}, we have

\begin{Theorem}\label{thm:thm3} (Extension of \cite[theorem 3]{KTW}.)
  Let $F$ be a codimension $d-1$ regular face of $\LRn$.
  Then there exists a BK-puzzle $P$ with $d$ labels, from which
  one can construct $d-1$ facets of $\LRn$ whose intersection is $F$.
\end{Theorem}

To characterize the BK-puzzles arising this way, we need to adapt the
``gentle loop'' technology of \cite{KTW} for Grassmannian puzzles.  
Orient the region edges as follows:
\begin{itemize}
\item If the edge is between a triangle and an adjacent rhombus,
  orient it toward the obtuse vertex of the rhombus.
\item If the edge is between an $(i,j)$-rhombus and an $(i,k)$-rhombus,
  $i>j>k$, orient it toward the obtuse vertex of the $(i,k)$-rhombus.
\item If the edge is between an $(i,k)$-rhombus and an $(j,k)$-rhombus,
  $i>j>k$, orient it toward the obtuse vertex of the $(i,k)$-rhombus.
\end{itemize}
Mnemonic: the rhombus with the greater spread takes precedence.
The BK-puzzle in figure \ref{fig:puzex1} has its region edges
oriented using this rule.

A \defn{gentle path} was defined in \cite{KTW}
as a path in this directed graph that, at each vertex,
either goes straight or turns $\pm 60^\circ$ (not $\pm 120^\circ$). 
For its generalization in this paper, we need an additional constraint:
if a vertex occurs as the intersection of two straight lines,
a gentle path through it {\em must go straight through.}
A \defn{gentle loop} is a gentle path whose first and last edges coincide
(edges not vertices -- the next turn after the last edge might not 
otherwise be gentle).

\begin{Proposition}\label{prop:prop2} (Extension of \cite[proposition 2]{KTW}.)
  Let $h = h_1\oplus h_2\oplus \cdots \oplus h_d$ be a clockwise overlay 
  of generic honeycombs, $P$ the corresponding BK-puzzle, 
  $\gamma = (\gamma_1,\ldots,\gamma_s)$ 
  a list of edges along a gentle path in $P$ of length $s>1$,
  and $\tilde\gamma$ the corresponding sequence of edges in $h$.
  Then the lengths of the honeycomb edges $(\tilde\gamma_i)$
  weakly decrease.
\end{Proposition}

To prove this, we will study the possible internal vertices in a BK-puzzle, 
and use this classification to simultaneously prove

\begin{Lemma}\label{lem:lem5} (Extension of \cite[lemma 5]{KTW}.)
  Let $P$ be a BK-puzzle without gentle loops, and $v$ an internal vertex.
  Label each region edge meeting $v$ with the number of gentle paths
  starting at that edge and terminating on the BK-puzzle boundary. 
  Then these labels are strictly positive, and $v$ has zero tension.
\end{Lemma}

\begin{proof}
  An internal vertex $v$ of a puzzle may be an obtuse vertex of
  (a priori) $0,1,2$ or $3$ rhombi. Consider the labels on the edges
  meeting $v$, clockwise: they strictly decrease across obtuse angles,
  stay the same across acute angles from triangles, and
  strictly increase across acute angles from rhombi.

  So if there are no obtuse vertices at $v$, there can be no acute
  rhombus vertices either, just six $i$-triangles for the same $i$.
  Such a $v$ has no gentle paths going through it.

  There cannot be three obtuse vertices at $v$, as that would have
  three strict increases with no room for any decreases.

  That leaves either $1$ or $2$ obtuse vertices at $v$. We draw the
  possibilities up to rotation and puzzle duality
  that have no triangles at $v$, only acute rhombi.
  (As each triangle makes the situation simpler we leave those
  cases to the reader.) In each case the labels are ordered $a>b>c>d>e$,
  or possibly $a>c>b>d$ in the second case.

  \centerline{  \epsfig{file=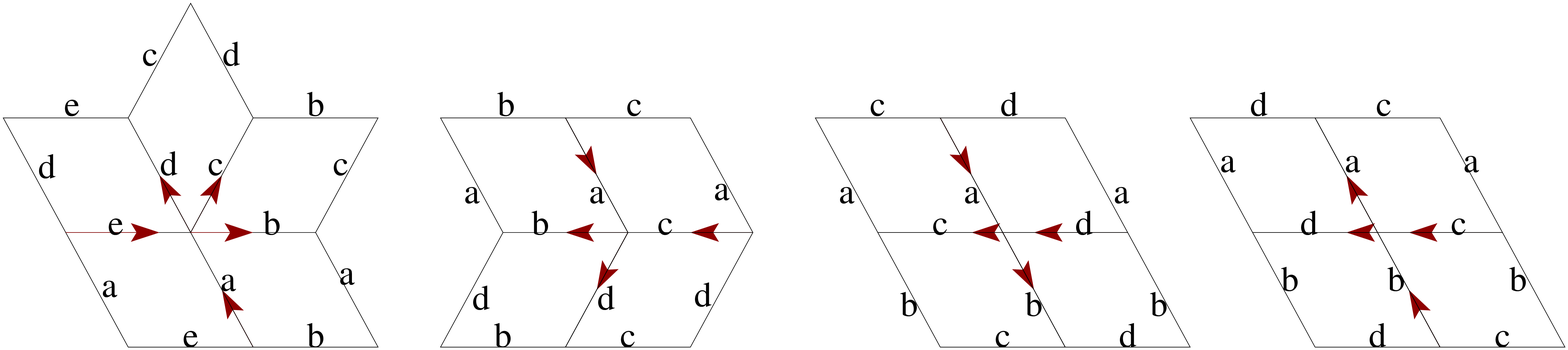,width=6.5in}}

  To prove the proposition, it suffices to check $s=2$. We draw the
  region (in the clockwise overlay) dual to the puzzle vertex.
  The arrows within are dual to $2$-step gentle paths.

  \centerline{  \epsfig{file=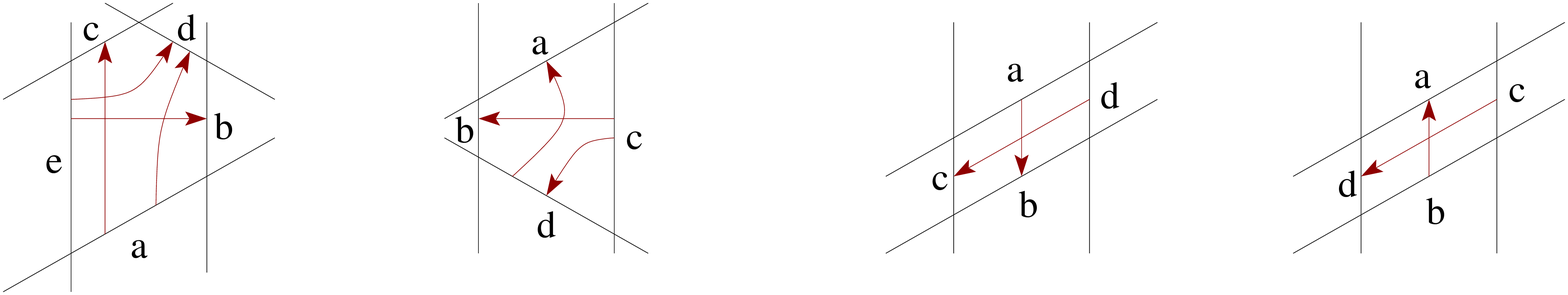,width=6.5in} }

  \begin{enumerate}
  \item None of these are ``gentle sinks'' -- any incoming path can be
    extended to be gently outgoing. This is why the number of gentle
    paths starting at an internal edge and terminating on the puzzle
    boundary is strictly positive.
  \item In each case, the length of an edge with an in-pointing arrow 
    is the sum of the lengths of the edges it points to.
    (This is where the modification of the \cite{KTW} definition 
    of ``gentle path'' is important.)
    Therefore if we change the length of 
    each honeycomb edge incident with a polygon to be
    the number of gentle paths emanating from it, possibly zero,
    the result is still a closed polygon.
    Dually, $v$ has zero tension.
  \qedhere
  \end{enumerate}
\end{proof}

\begin{Lemma}\label{lem:nogentleloops}(Extension of the corollary in
  \cite{KTW}.)
  BK-puzzles associated to clockwise overlays 
  have no gentle loops.
\end{Lemma}

\begin{proof}
  Let $\gamma = (\gamma_1,\ldots,\gamma_s)$ be a path whose corresponding 
  honeycomb edges all have the same length. Comparing to the BK-puzzle
  vertex classification above, every corresponding honeycomb region
  must be a parallelogram. (One possibility to remember is the
  leftmost vertex, but with $b=c=d$ and the rightmost two rhombi replaced
  with triangles.)  Chaining these together, the parallelograms all
  lie between the same two lines, so $\gamma$ can not close up to a loop.

  But by proposition \ref{prop:prop2}, any gentle loop $\gamma$ will
  have all corresponding honeycomb edges of the same length. So there
  can be no such loops.
\end{proof}

\begin{Proposition}\label{prop:prop3} (Extension of \cite[proposition 3]{KTW}.)
  Let $P$ be a BK-puzzle of size $n$ with no gentle loops. Then there
  exists a clockwise overlay $h$ such that 
  the BK-puzzle that lemma \ref{lem:lem4} associates to $h$ is $P$.
  By this lemma, the $d-1$ inequalities defined by the BK-puzzle determine
  a regular face of $\LRn$.
\end{Proposition}

The edges in $h$ are assigned lengths according to the number of 
gentle paths starting at their corresponding puzzle edges.
To know that the resulting $h$ is a {\em transverse} clockwise overlay
requires the strict positivity in lemma \ref{lem:lem5}.

\begin{Theorem}\label{thm:thm5} (Extension of \cite[theorem 5]{KTW}.)
  \begin{enumerate}
  \item There is a $1:1$ correspondence between BK-puzzles of size $n$ 
    without gentle loops and regular faces of $\LRn$.
  \item (An analogue of \cite[Theorem D]{RessayreGIT2}.)
    One face $F_1$ contains another, $F_2$, if the corresponding BK-puzzle $P_1$
    is an ambiguation $A_\sim P_2$ of the BK-puzzle $P_2$.
  \end{enumerate}
\end{Theorem}

\begin{proof}
  Claim (1) is a combination of lemma \ref{lem:lem3}, theorem \ref{lem:lem4},
  and proposition \ref{prop:prop3}.
  Claim (2) follows from lemma \ref{lem:lem4}.
\end{proof}

\begin{Remark}
The BK-puzzles for full flags (no repeated edge labels on a side) 
correspond to $2n$-dimensional regular faces. In the Hermitian sum context,
if we fix $\lambda$ and $\mu$, these become regular vertices and 
correspond to sums of \emph{commuting} Hermitian matrices. So they were
easy to study, historically, and people found many of the inequalities
on $\LRn$ by looking nearby these vertices.

One might hope, then, that every regular facet of $\LRn$ contains one of these
$2n$-dimensional regular faces.
A counterexample is provided by the unique puzzle in
$\Delta_{12112,12112}^{21121}$, as any attempt to disambiguate the $2$s 
inside a finer BK-puzzle breaks the $\inv_{ij}$ counts. Correspondingly,
in the overlay $h_1\oplus h_2$ pictured here, \\
\centerline{  \epsfig{file=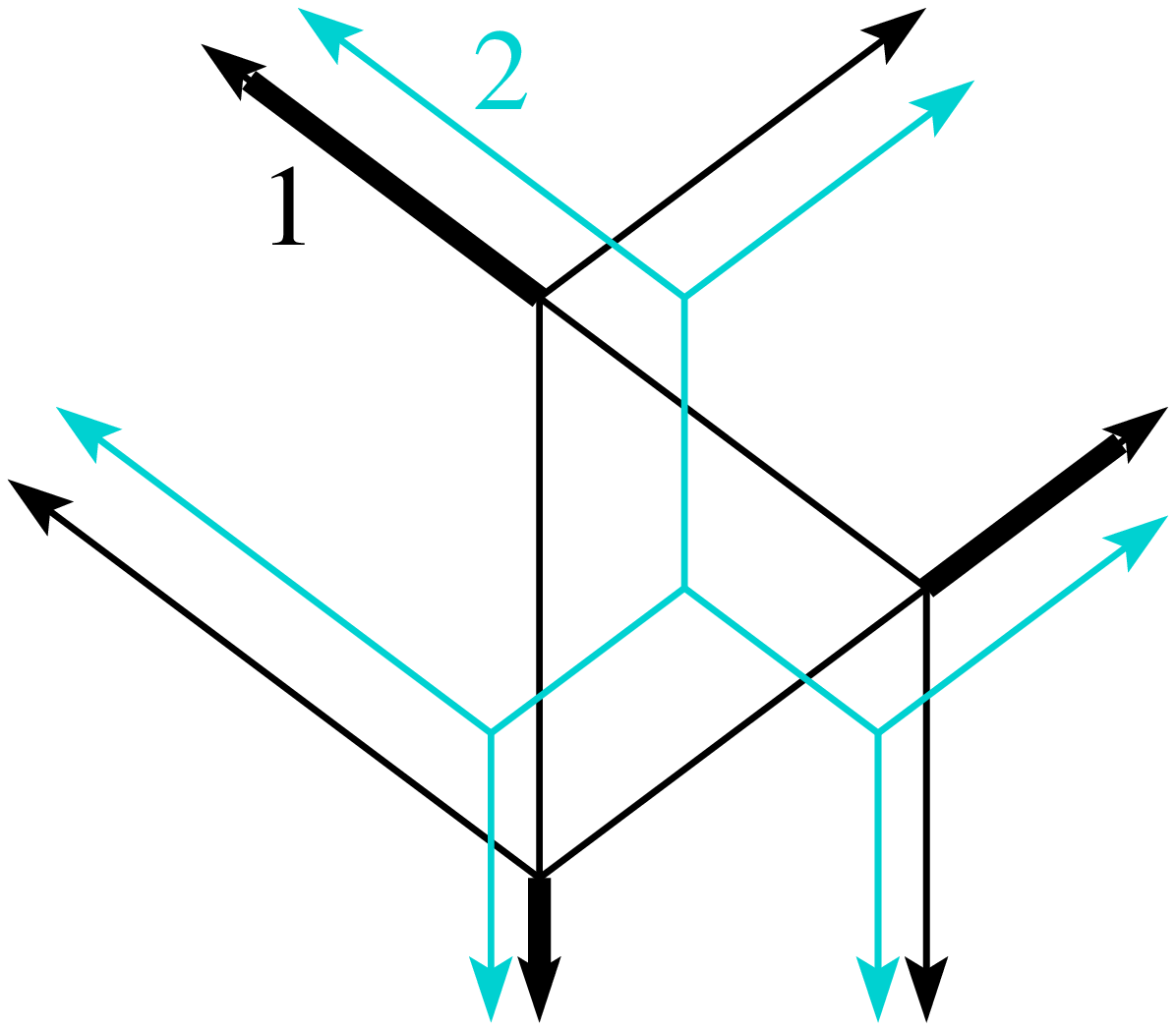,height=1.5in} }
none of the internal edges of $h_2$ can be shrunk to points while keeping
$h_1\oplus h_2$ transverse, as those edges all cross edges of $h_1$.
\end{Remark}

\begin{Theorem}\label{thm:thm67} (Extension of \cite[theorems 6 and 7]{KTW}.)
  A BK-puzzle is rigid iff it has no gentle loops.
\end{Theorem}

The hard direction, taking $3$ pages in \cite{KTW}, constructs a
new BK-puzzle $P'$ from a BK-puzzle $P$ and a minimal gentle loop.
We leave the reader to either check that those arguments generalize to
BK-puzzles, or to invoke \cite[Theorem C]{RessayreGIT2}.

The reader may be wondering about the redundant inequalities on $\LRn$
specified by nonrigid Grassmannian puzzles. Each one defines some face
of $\LRn$; what is the corresponding BK-puzzle? But \cite[theorem 8]{KTW}
says that these faces are never regular, so do not correspond to 
BK-puzzles.

\section{Rigid regular honeycombs}\label{sec:Fultonfaces}

As with puzzles, call a honeycomb \defn{rigid} if it is uniquely determined
by its boundary. 
These have received some study already;
under the deflation map linking honeycombs to puzzles,
these give the rigid puzzles indexing the regular facets of $\LRn$.
Fulton's conjecture (proven combinatorially in \cite{KTW} and
geometrically in \cite{RessayreFC,BKR}) is that an integral honeycomb
that is $\integers$-rigid is also $\reals$-rigid.

It is easy to see that the set of boundaries of rigid honeycombs 
is a union of faces of $\LRn$. In this theorem we characterize
which regular faces arise this way.

\begin{Theorem}\label{thm:rigidhoney}
  Let $h$ be a honeycomb such that $\partial h$ is regular.
  Then $h$ is rigid iff $h$ is 
  a clockwise overlay $h_1 \oplus \cdots \oplus h_d$ of honeycombs of
  size $1$ or $2$.
\end{Theorem}

\begin{proof}
  Theorem 2 from \cite{Hon1} says that some $h'$ 
  with $\partial h' = \partial h$ has simple degeneracies, and if we
  elide them (thinking of them as a sort of local overlay, rather than
  actual vertices), the underlying graph of the resulting space is acyclic. 
  Let $\{h'_i\}$ be the components of this forest, so $h' = \bigoplus_i h'_i$ 
  (not yet necessarily clockwise).  Each $h'_i$ automatically has simple 
  degeneracies, and being acyclic, has size $1$ or $2$.

  $\implies$ 
  Now assume $h$ is rigid, so $h=h'$. 
  Then for any pair $h'_i,h'_j$, one must be clockwise of the other,
  exactly as in the proof of \cite[lemma 3]{KTW}.
  Since $h'_i,h'_j$ must intersect, who is clockwise of whom is
  uniquely determined.

  Moreover, the clockwiseness relation must be transitive, or else there is
  some triple $h'_i$, $h'_j$, $h'_k$ with 
  $h'_i$ meeting $h'_j$ clockwise at $p$, 
  $h'_j$ meeting $h'_k$ clockwise at $q$, 
  $h'_k$ meeting $h'_i$ clockwise at $r$. But then \cite[lemma 2]{KTW} 
  can perturb $h$ along the unique loop from $p$ to $q$ to $r$ to $p$.
  
  Hence ``clockwiseness'' is a total order on the $\{h'_i\}$.

  $\Longleftarrow$ 
  We wish to show that if $\partial h' = \partial h$, then $h'=h$. 
  Each of the $d-1$ puzzle \emph{equalities} satisfied by $\partial h'$
  force $h'$ to be an overlay, since by the proof of \cite[theorem 2]{KTW}
  the edges in $h$ corresponding to rhombi must be length $0$.
  So $h'$ is an overlay of these honeycombs of size $1$ and $2$,
  each of which is rigid. That is, $h$ and $h'$ are the same overlay
  of the same size $1$ or $2$ honeycombs.
\end{proof}

\begin{Corollary}
  The rigid regular faces of $\LRn$ correspond to non-vanishing
  Belkale-Kumar coefficients on partial flag manifolds where each
  $k_i - k_{i-1} \leq 2$. (By the product formula, 
  such coefficients are automatically $1$.)
\end{Corollary}

Ressayre points out to us that this result is implicit
in the proof of \cite[Theorem 8]{RessayreCfree}.

\bibliographystyle{alpha}    

\end{document}